\documentclass[12pt]{article}
\usepackage{graphics,amsmath,amssymb}
\usepackage{amsthm}
\usepackage{amsfonts}
\usepackage{latexsym}
\usepackage{amscd}

\newtheorem{theorem}{Theorem}

\newtheorem{lemma}[theorem]{Lemma}

\begin{document}

\begin{center}
\title{Lack of Divisibility of ${2N \choose N}$ by three fixed odd primes infinitely often, through the Extension of a Result by P. Erd\H{o}s, et al.} 
\author{Robert J. Betts}
\maketitle
\emph{Graduate Department of Mathematics, University of Massachusetts Lowell, One University Avenue, Lowell, Massachusetts 01854 Robert\_Betts@alum.umb.edu}
\end{center}
\begin{abstract}
We provide a way to modify and to extend a previously established inequality by P. Erd\H{o}s, R. Graham and others and to answer a conjecture posed in the nineties by R. Graham, which bears on the lack of divisibility of the central binomial coefficient by three distinct, fixed odd primes. In fact the result will show by using an approach similar to their own which they proved for the case of two fixed odd primes, that the central binomial coefficient is not divisible infinitely often by three distinct and fixed odd primes. Therefore a generalization to more fixed odd primes than three but finite in number might be possible, at least if one is able to find some sufficient condition. The author hopes to answer this latter question in a subsequent paper.
\end{abstract}
\section{Introduction}
Let $p$, $q$ be two fixed and distinct odd positive primes, both smaller than positive integer $N$, and $A$, $B$ two other positive integers. P. Erd\H{o}s, R. Graham, I. Ruzsa and E. G. Straus~\cite{Erdos}, proved if the inequality
\begin{equation}
\frac{A}{p - 1} + \frac{B}{q - 1} \geq 1
\end{equation}
holds~\cite{Erdos} (See the Theorem, page 84), then there are infinitely many positive integers $N$ having base $p$ expansion~\cite{Kirch},~\cite{Rosen},~\cite{Knuth}~\cite{Warner}, with all digits equal to or smaller than $A$, and having base $q$ expansion with all the digits equal to or smaller than $B$. Their result~\cite{Erdos} (See page 84) depends on both \(\log p, \log q\) being \emph{incommensurable numbers} (See next Section for definitions). Their proof of the inequality then establishes that~\cite{Encyclopedia},~\cite{Erdos},~\cite{Sarkozy},~\cite{Everett} 
\begin{equation}
\gcd\left({2N \choose N}, pq\right) = 1,
\end{equation}
is true infinitely often. In Section 3 we extend and modify their result in part by using three odd primes from which one derives three incommensurable numbers instead of two, to extend the prior result for three positive integers \(A, B, C, A \leq B \leq C\).

\indent Years ago, possibly as early as the mid-nineteen nineties, R. Graham~\cite{Erdos} conjectured that 
$$
\gcd\left({2N \choose N}, 105\right) = 1, \: 105 = 3\times5\times7,
$$
is true for infinitely many positive integers $N$. This would mean if $N$ is very large, then ${2N \choose N}$ would fail to be divisible infinitely often by numbers of the form $3^{a}5^{b}7^{c}$ where \(3, 5, 7\) are three very small consecutive twin primes and \(a, b, c\) are three very large positive integer exponents. Unfortunately many of the original URLs on which the conjecture was posted, have been consigned to the obscure ether of inactive websites. The author recalls having visited one of these websites between 2000-2002 when he was a mathematics undergraduate, using at the time a Windows 98 machine. In this paper we prove the more general result 
$$
\gcd\left({2N \choose N}, pqr\right) = 1,
$$
holds infinitely often, for three fixed, distinct odd primes $p$, $q$, $r$, where \(p < q < r\). This result in the next Section depends on a lemma (next Section) and on the prior methods used by P. Erd\H{o}s for the two integers \(p, q\). One cannot emphasize stongly enough that one might use a more general method to derive the stronger result of four or more odd primes, that is, for some \(n > 3\) fixed odd primes \(p_{1}, p_{2}, p_{3}, \ldots, p_{n}\), even if one and at the very least, should require in the more general result the establishment of some necessary or sufficient condition or both, for which ${2N \choose N}$ will not be divisible infinitely often by \(n > 3\) distinct odd primes.
\section{Ansatz}
The approach in this paper follows very closely and modifies slightly, the prior approach and methods of P. Erd\H{o}s et al., in the proof to Theorem 1 in their paper~\cite{Erdos} (pages 84-86), which we use here as an Ansatz. 

\indent First we will provide some definitions and two lemmas.

\noindent DEFINITION 1: Let $P$, $J$ be two positive integers with \(P < N\). Then a positive integer $N$ is called a $(P, J)$-good number if, for integers \(a_{d}, a_{d - 1}, \cdots, a_{0}\),
\begin{eqnarray}
N&=&a_{d}P^{d} + a_{d - 1}P^{d - 1} + \cdots + a_{0}\nonumber \\
 &\Longrightarrow& max\{a_{d}, a_{d - 1}, \cdots, a_{0}\} \leq J,
\end{eqnarray}
where~\cite{Ireland}
\begin{equation}
a_{d} \not = 0.
\end{equation}
Let $P$ be an odd prime. The $P$-adic valuation \(\nu_{P}(N) = t \geq 1\) is the largest integer exponent $t$ for which $P^{t}$ divides $N$. If \(P^{t} \not | N\) then \(t = 0\) and \(|N|_{P} = 1\). So if $P^{t}$ does not divide $N!$ we have \(\nu(N!) = 0\). The nonzero integer $a_{d}$ along with the other digits \(a_{d - 1}, \ldots, a_{0}\) are such that  
$$
\{a_{d}, a_{d - 1}, \ldots, a_{0}\}
$$
are elements of the finite field \(\mathbb{Z}_{P} \subset \mathbb{Q}_{P}\) if $P$ is an odd prime. Actually we have \(\mathbb{Z}_{P} = \mathbb{Z}/P\mathbb{Z}\) (up to isomorphism). Further we have that
$$
a_{d} = N \mod P^{d}.
$$
For the sum 
$$
s_{P}(N) =  \sum_{j = 0}^{d}a_{j},
$$
we have the result by Legendre~\cite{Straub}
$$
\nu_{P}(N!) = \frac{N - s_{P}(N)}{P - 1},
$$
where $\nu_{P}(N!)$ is the $P$-adic valuation of $N!$, where $N!$ would have the $P$-adic norm $|N!|_{P}$. 

\indent It follows already from the inequality in Eqtn. (1) proved by P. Erd\H{o}s, R. Graham, I. Ruzsa and E. G. Straus~\cite{Erdos} that there are infinitely many positve integers $N$ that are both $(p, A)$-good numbers and $(q, B)$-good numbers.

\indent The definition of a \emph{commensurable number} has been expressed differently over the past century~\cite{Wheeler},~\cite{Planetmath}. Therefore for our purposes we shall use the following Definition.

\noindent DEFINITION 2: Let \(\varepsilon_{1}, \varepsilon_{2}\) be unequal real numbers. Then \(\log \varepsilon_{1}, \log \varepsilon_{2}\) are called \emph{commensurable} numbers if their ratio is a rational number. Otherwise they are called \emph{incommensurable} numbers.

\indent With the three numbers \(p, q, r\) being three distinct odd primes, 
$$
\log p, \log q, \log r
$$
are incommensurable numbers. The result by Erd\H{o}s et al. uses the fact that for the integers \(p, q\) as primes and for two positive integer powers \(\alpha, \beta\) with \(p^{\alpha}, q^{\beta}\) both positive integers, $\log p$ and $\log q$ both are incommensurable numbers.
\begin{lemma}
For three fixed and distinct odd primes $p$, $q$, $r$, \(p < q < r\), let $\log p$, $\log q$ and $\log r$ all be incommensurable and \(A, B, C\) three positive integers with \(A \leq B \leq C\). Then there are infinitely many positive integer exponents \(\alpha, \beta, \Gamma, \Gamma^{\prime}\), such that 
\begin{eqnarray}
& &\left|p^{\alpha} - \frac{B}{2} \frac{q^{\beta} - 1}{q - 1}\right| < \frac{B}{2} \frac{q^{\beta} - 1}{q - 1},\\
& &\left|q^{\beta} - \frac{C}{2} \frac{r^{\Gamma} - 1}{r - 1}\right| < \frac{C}{2} \frac{r^{\Gamma} - 1}{r - 1},\\
& &\left|p^{\alpha} - \frac{C}{2} \frac{r^{\Gamma^{\prime}} - 1}{r - 1}\right| < \frac{C}{2} \frac{r^{\Gamma^{\prime}} - 1}{r - 1},
\end{eqnarray}
where the base $q$ expansion of $p^{\alpha}$ either has all its digits (or coefficients) equal to or smaller than $B$ or else has one digit (or coefficient) smaller than $B$ before one that is greater than $B$, and where the base $r$ expansion of $q^{\beta}$ and $p^{\alpha}$, respectively, either has all its digits (or coefficients) equal to or smaller than $C$ or else has one digit (or coefficient) smaller than $C$ before one that is greater than $C$. 
\end{lemma}
\begin{proof}
The previous authors cite the first inequality in Eqtn. (5). The inequality in Eqtn. (6) follows by replacing \(p, q\) with \(q, r\) and \(\alpha, \beta\) with \(\beta, \Gamma\). The inequality in Eqtn. (7) follows by replacing \(p, q\) with \(p, r\) and \(\alpha, \beta\) with \(\alpha, \Gamma^{\prime}\).
\end{proof}
\begin{lemma}
Let $a$ be any positive integer. Then the interval 
\begin{equation}
\left[a, \left(\frac{p - 1}{A}\right)a\right) \subset \mathbb{R}^{1}
\end{equation}
contains a $(p, A)$-good number.
\end{lemma}
\begin{proof}
The previous authors provide this proof~\cite{Erdos} (pages 85-86).
\end{proof}
We provide in brief a synopsis of the approach and methods used by the previous authors cited. They prove there are infinitely many $(p, A)$-good and $(q, B)$-good numbers by proving the inequality
\begin{equation}
\frac{A}{p - 1} + \frac{B}{q - 1} \geq 1.
\end{equation}
they do this by proving that, for a given $(p, A)$-good number 
\begin{equation}
N = a_{n}p^{n} + a_{n - 1}p^{n - 1} + \cdots + a_{m}p^{m}, 
\end{equation}
$$
n > m \geq 0,
$$
there exists an integer $N^{*}$ with a base $q$ expansion (we alter some of the index numbers the previous four authors used here for our own purposes)
\begin{equation}
N^{*} = b_{\rho}q^{\rho} + b_{\rho - 1}q^{\rho - 1} + \cdots + b_{i + 1}q^{i + 1} + b_{i}^{*}q^{i} + \cdots
\end{equation}
where as they state~\cite{Erdos} (See page 84), either \(b_{i}^{*} = b_{i}\) and \(N^{*} < N\) hold or else \(B > b_{i}^{*} > b_{i}\) or \(B = b_{i}^{*}\), where the first digit or coefficient with index $< i$ in the expansion that does not equal $B$ is smaller than $B$. The index $i$ also appears in their base $q$ expansion of 
\begin{equation}
N = b_{\rho}q^{\rho} + b_{\rho - 1}q^{\rho - 1} + \cdots + b_{i}q^{i} + \cdots + b_{j}q^{j} + \cdots
\end{equation}
where as they describe, \(b_{j} > B\) is true for some largest index \(j < i\), and where $i$ is the smallest so that \(b_{i} < B\). Let
\begin{equation}
T = b_{i - 1}q^{i - 1} + b_{i - 2}q^{i - 2} + \cdots + b_{0}.
\end{equation}
They call this integer the ``tail" of the base $q$ expansion for $N$. To obtain a smaller $(p, A)$-good number one must bound some suitable integer $S$ above by $p^{m}$, where
\begin{equation}
S = p^{m} - A\left(\frac{p^{m} - 1}{p - 1}\right).
\end{equation}
This follows because $\log p$ is an incommensurable number. Then if \(T \geq S\), \(N^{*} = N - S\) will be the sought after number. In the case \(T < S\) they prove the existence of a $(p, A)$-good number $U$ such that
\begin{equation}
q^{i} - T \leq U \leq q^{i} - T + B\left(\frac{q^{i} - 1}{q - 1}\right)
\end{equation}
\begin{equation}
\leq \frac{p - 1}{A}\left(\frac{A}{p - 1} + \frac{B}{q - 1}\right)(q^{i} - T)  
\end{equation}
\begin{equation}
\geq \left(\frac{p - 1}{A}\right)a.      
\end{equation}
This means 
\begin{equation}
U \in \left[a, \left(\frac{p - 1}{A}\right)a\right) \Longrightarrow N^{*} = N + U. 
\end{equation}
Look again at Eqtn. (11). that part of $N^{*}$ for which the digits have indices from $i$ to $\rho$ is a $(p, B)$-good number, while the part with indices from $i - 1$ down to $0$ is smaller than the $(p, A)$-good number $S$, which is greater than $T$, and so is a $(p, A)$ good number. It follows then that
$$
N = N^{*} - S
$$
is both a $(p, A)$-good number and a $(q, B)$-good number. That is, since the inequality in Eqtn. (9) holds (See Eqtn. (16)) and since $U$ is bounded above by $p^{m}$, $N$ must be $(q, B)$-good as well as $(p, A)$-good. 

\indent In the next Section we modify this argument to handle the three primes \(p, q, r\) for three incommensurable numbers \(\log p, \log q, \log r\) instead of two.
\section{The Theorem}
In this Section we present the main result. Before we state and prove Theorem 1 our position is as follows: Let $N$ be \emph{some} positive integer that is both $(p, A)$-good and $(q, B)$-good (by the emphasis on the word ``some" we mean we do not make the claim about any $(p, A)$-good and $(q, B)$-good number in general), where \(\log p, \log q\) are both incommensurable numbers. Then with \(r > q\) some third odd prime smaller than $N$ where \(\log q, \log r\) are also incommensurable, Inequality (6) holds in Lemma 1 as well as Inequality (7) where this latter inequality holds also, since the pair \(\log p, \log r\) are also incommensurable.    \begin{theorem}
Let \(A, B, C, A \leq B \leq C\) be three positive integers with \(p/2 \leq A, q/2 \leq B, r/2 \leq C\), $a$ any positive integer and \(p, q, r, p < q < r\) three fixed, distinct odd primes each smaller than $N$, where $N$ is both a $(p, A)$-good number and a $(q, B)$-good number, with the corresponding base $p$ and base $q$ expansions
\begin{eqnarray}
N&=&a_{n}p^{n} + a_{n - 1}p^{n - 1} + \cdots + a_{m}p^{m}, \: n > m \geq 0, \\
N&=&b_{\rho}q^{\rho} + b_{\rho - 1}q^{\rho - 1} + \cdots + b_{\kappa}q^{\kappa}, \: \rho > \kappa \geq 0,
\end{eqnarray}
\begin{eqnarray}
n&>&m,\nonumber \\
\rho&\geq&\kappa.
\end{eqnarray}
Then there exist three positive integers \(N^{\prime}, N^{\prime \prime}, N^{*}\), the first two with, respectively, some base $r$ expansions and $N^{*}$ with some base $q$ expansion and a $(q, B)$-good number $U^{\prime}$, a $(p, A)$-good number $U^{\prime \prime}$ also a $(p, A)$-good number $U$, such that all of the following conditions hold:
\begin{enumerate}
\item
\begin{equation}
\frac{B}{q - 1} + \frac{C}{r - 1} \geq 1,
\end{equation}
\begin{equation}
U^{\prime} \in \left[a, \left(\frac{q - 1}{B}\right)a\right)
\end{equation}
\begin{equation}
N^{\prime} = N + U^{\prime},
\end{equation}
such that the $(p, A)$-good and $(q, B)$-good number $N$ also is $(q, B)$-good and $(r, C)$-good. \\
\item
\begin{equation}
\frac{A}{p - 1} + \frac{C}{r - 1} \geq 1,
\end{equation}
\begin{equation}
U^{\prime \prime} \in \left[a, \left(\frac{p - 1}{A}\right)a\right)
\end{equation}
\begin{equation}
N^{\prime \prime} = N + U^{\prime \prime},
\end{equation}
such that the $(p, A)$-good and $(q, B)$-good number $N$ also is $(p, A)$-good and $(r, C)$-good. \\
\item
\begin{equation}
\frac{A}{p - 1} + \frac{B}{q - 1} \geq 1,
\end{equation}
\begin{equation}
N^{*} = N + U,
\end{equation}
\begin{equation}
U \in \left[a, \left(\frac{p - 1}{A}\right)a\right).
\end{equation}
\\
\item
\begin{equation}
\frac{A}{p - 1} + \frac{B}{q - 1} + \frac{C}{r - 1} \geq 1.
\end{equation}
\end{enumerate}
\end{theorem}
\begin{proof}
We will consider the three numbers \(A, B, C\) to be such that \(\phi(p)/2 < p/2 \leq A, \phi(q)/2 < q/2 \leq B, \phi(r)/2 < r/ 2 \leq C\), where ``$\phi$" is Euler's totient function. First we shall prove Condition (1) (See Eqtns. (22)--(24). We do this basically by interchanging the roles of \(p, A\) and \(q, B\) used by the previous authors~\cite{Erdos}, to those of \(q, B\) and \(r, C\), for the same integer $N$, which already is given here to be both a $(p, A)$-good number and a $(q, B)$-good number. 

\indent First it follows by definition that for nonnegative integers \(\alpha, \beta, \Gamma, \Gamma^{\prime}\) and for three distinct, fixed odd primes \(p, q, r\), the three numbers \(\log p, \log q, \log r\) are incommensurable numbers. Therefore we restrict all exponent powers of these primes in some odd prime base expansion of $N$ as being nonnegative integers. 

\indent Since $N$ is given as a $(p, A)$-good and a $(q, B)$-good number, with the corresponding base $p$ and base $q$ expansions being
\begin{eqnarray}
N&=&a_{n}p^{n} + a_{n - 1}p^{n - 1} + \cdots + a_{m}p^{m}, \: n > m \geq 0, \\
N&=&b_{\rho}q^{\rho} + b_{\rho - 1}q^{\rho - 1} + \cdots + b_{\kappa}q^{\kappa}, \: \rho > \kappa \geq 0,
\end{eqnarray}
\begin{eqnarray}
n&>&m,\nonumber \\
\rho&\geq&\kappa,
\end{eqnarray}
let
\begin{equation}
N = c_{\gamma}r^{\gamma} + c_{\gamma - 1}r^{\gamma - 1} + \cdots + c_{\iota + 1}r^{\iota + 1} + c_{\iota}r^{\iota} + \cdots + c_{\lambda}r^{\lambda} + \cdots
\end{equation}
be the  base $r$ expansion of $N$, where as the previous four authors defined similarly~\cite{Erdos} (Page 84. Compare the definition of the integer $N$ in the LEMMA.), the index $\iota$ is the least index greater than $\lambda$ for which \(c_{\iota} < C\) is true and $\lambda$ is the largest index for which \(c_{\lambda} > C\).

\indent Let, for positive integer $N^{\prime}$ greater than \(p, q, r\),
\begin{equation}
N^{\prime} = c_{\gamma}r^{\gamma} + c_{\gamma - 1}r^{\gamma - 1} + \cdots + c_{\iota + 1}r^{\iota + 1} + c_{\iota}^{\prime}r^{\iota} + \cdots
\end{equation}
Let
\begin{equation}
T^{\prime} = N - (c_{\gamma}r^{\gamma} + c_{\gamma - 1}r^{\gamma - 1} + \cdots + c_{\iota + 1}r^{\iota + 1} + c_{\iota}r^{\iota}).
\end{equation}
If we can subtract from $N$ some smallest $(q, B)$-good number bounded above by $T^{\prime}$ to get some smaller $(q, B)$-good number we would get \(N^{\prime} < N\). A $(q, B)$-good number $U^{\prime}$ smaller than this requires some smallest number equal to~\cite{Erdos} (Compare the bottom of page 84)
\begin{eqnarray}
S^{\prime}&=&q^{\kappa} - B\left(\frac{q^{\kappa} - 1}{q - 1}\right)\\
          &=&\left(\frac{q - B - 1}{q - 1}\right)(q^{\kappa} - 1) + 1,
\end{eqnarray}
which one obtains only after a little algebra. Expanding out the product in Eqtn. (39) derives
$$
\left(\frac{q - B - 1}{q - 1}\right)(q^{\kappa - 1} + q^{\kappa - 2} + \cdots + 1) + 1,
$$
which, clearly, is a $(q, B)$-good number since every coefficient is smaller than $B$.

\indent Suppose the value of $S^{\prime}$ is at most $T^{\prime}$. Then \(N^{\prime} = N - S^{\prime}\) and since we are given that $N$ is $(q, B)$-good, \(N = N^{\prime} + S^{\prime}\) is not only a $(q, B)$-good number but also an $(r, C)$-good number since all it coefficients are smaller than $C$ (meaning also that \(c_{\iota}^{\prime} = c_{\iota}\) since \(N^{\prime} < N\), where $S^{\prime}$ is bounded above by the base $r$ expansion $T^{\prime}$). If \(S^{\prime} > T^{\prime}\) we use an alternative approach. Since from the definition of $\iota$ we have
\begin{equation}
c_{\iota - 1} \geq C, \cdots c_{\lambda} \geq C, \cdots, c_{0} \geq C,\\
\end{equation}
we have the inequality
\begin{eqnarray}
S^{\prime} > T^{\prime}&>&C(r^{\iota - 1} + r^{\iota - 2} + \cdots + 1)\\
                       &\geq&C\left(\frac{r^{\iota} - 1}{r - 1}\right) + 1.
\end{eqnarray}
From this one derives the inequality (See Eqtns. (38)--(42))
\begin{equation}
r^{\iota} - 1 < \left(\frac{r - 1}{C}\right)\left(\frac{q - B - 1}{q - 1}\right)(q^{\kappa} - 1).
\end{equation}
Let \(N^{\prime} = N + U^{\prime}\), where
\begin{equation}
r^{\iota} - T^{\prime} \leq U^{\prime} \leq r^{\iota} - T^{\prime} + C\left(\frac{r^{\lambda - 1}}{r - 1}\right).
\end{equation}
It follows that \(c_{\iota}^{\prime} = c_{\iota} + 1\). Now if $N^{\prime}$ is not yet an $(r, C)$-good number let some index $\iota^{\prime}$ be such that \(\iota^{\prime} \leq \iota + 1\), such that \(c_{\iota^{\prime}}^{\prime} < C\) while \(c_{\lambda^{\prime}}^{\prime} > C\) holds for some index \(\lambda^{\prime} < \iota^{\prime}\). Then when \(a = r^{\iota} - T^{\prime}\),
\begin{eqnarray}
a&\leq&r^{\iota} - 1 - C\left(\frac{r^{\iota - 1}}{r - 1}\right) = \left(\frac{r - C - 1}{r - 1}\right)(r^{\iota} - 1)\\
 &\leq&B\left(\frac{r^{\iota} - 1}{q - 1}\right).
\end{eqnarray}
The previous authors already showed that
\begin{equation}
a \leq A\left(\frac{q^{i} - 1}{p - 1}\right),
\end{equation}
where $i$ is the index in Eqtns. (11)--(13), where they proved that \(N = N^{*} - U\) is both a $(p, A)$-good number and a $(q, B)$-good number for some $(p, A)$-good number \(U \in [a, (p - 1)a/A)\). Then
\begin{eqnarray}
& &r^{\iota} - T^{\prime} + C\left(\frac{r^{\iota} - 1}{r - 1}\right)\\                                                                 
&\geq&\left(1 + \frac{q - 1}{B} \cdot \frac{C}{r - 1}\right)(r^{\iota} - T^{\prime})\nonumber \\
&=&\left(\frac{q - 1}{B}\right)\left(\frac{B}{q - 1} + \frac{C}{r - 1}\right)(r^{\iota} - T^{\prime}) \geq \left(\frac{q - 1}{B}\right)a
\end{eqnarray}
\begin{eqnarray}
&\Longrightarrow&\left(\frac{q - 1}{B}\right)a \leq r^{\iota} - 1 < \left(\frac{r - 1}{C}\right)\left(1 - \frac{B}{q - 1}\right)(q^{\kappa} - 1)\\
&\leq           &\left(\frac{r - 1}{C}\right)\left(\frac{C}{r - 1}\right)(q^{\kappa} - 1)\nonumber \\
&=              &q^{\kappa} - 1.\nonumber \\
&\Longrightarrow&U^{\prime} \in \left[a, \frac{q - 1}{B}a\right) \subset [a, q^{\kappa}).
\end{eqnarray}
It follows then that, in Eqtn. (49),
\begin{equation}
\frac{B}{q - 1} + \frac{C}{r - 1} \geq 1.
\end{equation}
Therefore we have proved Condition (1) (See Eqtns. (22)--(24)). This proves the integer $U^{\prime}$ is a $(q, B)$-good number. Now that part of $N^{\prime}$ that has all digits (or coefficients) with indices from $\iota$ all the way up to the index \(\gamma \geq \iota\) (See Eqtn. (36) and Eqtn. (37)) is an $(r, C)$-good number, while that part of $N^{\prime}$ with digits (or coefficients) having indices from $\iota - 1$ to $0$ is bounded above by the $(q, B)$-good number \(S^{\prime} > T^{\prime}\), where $U^{\prime}$ also is a $(q, B)$-good number. Since the inequality in Eqtn. (49) holds, it must be that \(N = N^{\prime} - U^{\prime}\) is also an $(r, C)$-good number as well as a $(p, A)$-good number and a $(q, B)$-good number (which was given), since with the ``tail" \(T^{\prime} < S^{\prime} < p^{m}\) all its digits or coefficients are equal to or smaller than $C$. 

\indent Now $N$ already was given to be a $(p, A)$-good number as well as a $(q, B)$-good number. Thus it follows from the above argument it must be $(p, A)$-good, $(q, B)$-good as well as $(r, C)$-good. In fact since 

$$
\frac{A}{p - 1} + \frac{B}{q - 1} \geq 1,
$$
is true, we must have actually that 
$$
\frac{A}{p - 1} + \frac{B}{q - 1} + \frac{C}{r - 1} \geq 1,
$$
due to the inequality in Eqtn. (52). But we strengthen this latter result with the following additional argument.

\indent Next we prove Condition 2. We proceed as before when we showed the integer $N$ was both a $(q, B)$-good number as well as an $(r, C)$-good number, when we were given $N$ was a $(p, A)$-good number and a $(q, B)$-good number.  This time we interchange the roles~\cite{Erdos} of \(p, A\) and \(q, B\) with those of \(p, A\) and \(r, C\). 

\indent Let the base $r$ expansions of the integers $N$ and $N^{\prime}$ and the integer $T^{\prime}$ remain as they were defined in Eqtns. (35)--(37). Since the smallest possible number to subtract from $N$ in order to find a smaller $(p, A)$-good number was found by the previous four authors to be $S$ (See Eqtn. (14)), let \(S^{\prime \prime} \geq S\) be some $(p, A)$-good number. Then if \(S \leq S^{\prime \prime} \leq T^{\prime}\) holds, we have found the required number that is $(p, A)$-good, $(q, B)$-good and $(r, C)$-good for some integer \(N^{\prime} = N - S^{\prime \prime}\). Otherwise suppose \(S^{\prime \prime} > T^{\prime}\). Then similarly as before we arrive at the inequality
\begin{eqnarray}
S^{\prime \prime} > T^{\prime}&>&C(r^{\iota - 1} + r^{\iota - 2} + \cdots + 1)\\
                              &\geq&C\left(\frac{r^{\iota} - 1}{r - 1}\right) + 1.
\end{eqnarray}
Surely this integer $S^{\prime \prime}$ must exist. All one need do is find the suitable value for the positive integer exponent $m$ in
\begin{eqnarray}
\log S^{\prime \prime}&\geq&\log \left(p^{m} - A\left(\frac{p^{m} - 1}{p - 1}\right)\right)\\
                      &=&\log\left(\left(\frac{p - A - 1}{p - 1}\right)(p^{m} - 1) + 1\right),
\end{eqnarray}
for suitable positive integer exponent $\iota$ in Eqtns. (53)--(54). From this one derives the inequality (See Eqtns. (38)--(42) and Compare Eqtn. (43))
\begin{equation}
r^{\iota} - 1 < \left(\frac{r - 1}{C}\right)\left(\frac{p - A - 1}{p - 1}\right)(p^{m} - 1).
\end{equation}
Then if we let \(N^{\prime} = N + U^{\prime \prime}\), where \(U^{\prime \prime} \geq U^{\prime}\) and where 
\begin{equation}
r^{\iota} - T^{\prime} \leq U^{\prime \prime} \leq r^{\iota} - T^{\prime} + C\left(\frac{r^{\lambda - 1}}{r - 1}\right).
\end{equation}
Then if while we replace in Eqtns. (44)--(52), $U^{\prime}$ with \(U^{\prime \prime} \geq U\) such that \(U^{\prime \prime}, U \in [a, (p - 1)a/A)\) (See Eqtn. (30)) for some exponents \(\iota, m\) such that Eqtns. (53)--(56) are true, we also change in all these inequalities for \(B \rightarrow A, q \rightarrow p, \kappa \rightarrow m\), we obtain the inequality
\begin{eqnarray}
& &r^{\iota} - T^{\prime} + C\left(\frac{r^{\iota} - 1}{r - 1}\right)\\                                                                 
&\geq&\left(1 + \frac{p - 1}{A} \cdot \frac{C}{r - 1}\right)(r^{\iota} - T^{\prime})\nonumber \\
&=&\left(\frac{p - 1}{A}\right)\left(\frac{A}{p - 1} + \frac{C}{r - 1}\right)(r^{\iota} - T^{\prime}) \geq \left(\frac{p - 1}{A}\right)a
\end{eqnarray}
\begin{eqnarray}
&\Longrightarrow&\left(\frac{p - 1}{A}\right)a \leq r^{\iota} - 1 < \left(\frac{r - 1}{C}\right)\left(1 - \frac{A}{p - 1}\right)(p^{m} - 1)\\
&\leq           &\left(\frac{r - 1}{C}\right)\left(\frac{C}{r - 1}\right)(p^{m} - 1)\nonumber \\
&=              &p^{m} - 1.\nonumber \\
&\Longrightarrow&U^{\prime \prime} \in \left[a, \frac{p - 1}{A}a\right) \subset [a, p^{m}).
\end{eqnarray}
From this and from Eqtn. (60) one gets the inequality, similar to what we found in Eqtn. (52), 
\begin{equation}
\frac{A}{p - 1} + \frac{C}{r - 1} \geq 1.
\end{equation}
Then when \(N = N^{\prime} - U^{\prime \prime}\) for suitable integer exponent $m$ in Eqtns. (55)--(56) we have found that $N$ is $(p, A)$-good and $(r, C)$-good as well as $(q, B)$-good and $(r, C)$ good. This proves Condition 2 (Eqtns. (25)--(27)). 

\indent Next we consider Condition 3, which already has been proved by P. Erd\H{o}s, R. Graham, I. Ruzsa, E. Straus~\cite{Erdos}, from which it follows automatically all the three inequalities
\begin{eqnarray}
\frac{B}{q - 1} + \frac{C}{r - 1}&\geq&1,\\
\frac{A}{p - 1} + \frac{C}{r - 1}&\geq&1,\nonumber \\
\frac{A}{p - 1} + \frac{B}{q - 1}&\geq&1\nonumber 
\end{eqnarray}
must hold, meaning the inequality in Condition 4, namely
\begin{equation}
\frac{A}{p - 1} + \frac{B}{q - 1} + \frac{C}{r - 1} \geq 1,
\end{equation}
also must be true.
\end{proof}
It goes without saying that we have in Eqtns. (32)--(35),
\begin{eqnarray}
&               &n, \rho, \gamma \in \mathbb{N}\\
&\Longrightarrow&n \geq \rho \geq \gamma.
\end{eqnarray}

\indent In essence what we have done in the Theorem is the following. Suppose $N$ is some $(p, A)$-good and $(q, B)$-good number. In the original argument~\cite{Erdos} replace their base $q$ expansions for \(N, N^{*}\) with suitable base $r$ expansions \(N, N^{\prime}\) (See Eqtns. (35)--(37)). Then show the integers \(S^{\prime}, U^{\prime}\) exist such that the inequality in Eqtn. (52) holds, so that the $(p, A)$-good and $(q, B)$-good number $N$ is shown to be $(q, B)$-good and $(r, C)$-good. Next for some suitable positive integer exponents $m$ and $\iota$ in Eqtns. (52)--(56), we show the integers \(S^{\prime \prime}, U^{\prime \prime}\) exist, so that the inequality in Eqtn. (63) holds, to show that the $(p, A)$-good and $(q, B)$-good number $N$ is a $(p, A)$-good number and an $(r, C)$-good number. 

\indent When \(k = 2\) in~\cite{Straub} (See Section 1 and Section 5)
$$
(1 - k^{2}x)^{-1/k} = \sum_{N \geq 0}c(N, k)x^{N},
$$
the coefficients $c(N, 2)$ are equal to ${2N \choose N}$. Our result indicates that 
$$
\nu_{p}\left({2N \choose N}\right) = \nu_{q}\left({2N \choose N}\right) = \nu_{r}\left({2N \choose N}\right) = 0    
$$
$$
\Longrightarrow  \left|{2N \choose N}\right|_{p} = \left|{2N \choose N}\right|_{q} = \left|{2N \choose N}\right|_{r} = 1
$$
is true infinitely often. There is a proof that shows why \(p/2 \leq A, q/2 \leq B, r/2 \leq C\) must hold~\cite{Straub} (See Theorem 3.6).

\indent Let $A(N)$ denote the least integer that does not divide ${2N \choose N}$~\cite{Erdos} (Page 91). The previous authors have stated that for \(\epsilon > 0\) the inequality 
\begin{equation}
exp((\log N)^{1/2 - \epsilon}) < A(N) < exp((\log N)^{1/2 + \epsilon}),
\end{equation}
shows $A(N)$ to be bounded for sets of nonzero asymptotic density. One can test this for \(N = 10, N = 756, N = 757\) (See Section 5), since for \(A(10) = A(756) = A(757) = 3\),
\begin{eqnarray}
& &exp((\log 10)^{1/2 - \epsilon}) < 3 < exp((\log 10)^{1/2 + \epsilon}),\\
& &exp((\log 756)^{1/2 - \epsilon}) < 3 < exp((\log 756)^{1/2 + \epsilon}),\nonumber \\
& &exp((\log 757)^{1/2 - \epsilon}) < 3 < exp((\log 757)^{1/2 + \epsilon}),
\end{eqnarray}
where \(\epsilon = 1/2\). One of course can find better values for $\epsilon$. We include the inequality here only for purely illustrative purposes.
\section{Geometrical Interpretation of the Inequality \(\frac{A}{p - 1} + \frac{B}{q - 1} + \frac{C}{r - 1} \geq 1\) on $\mathbb{R}^{3}$}
Consider the relation
\begin{equation}
f: \mathbb{R}^{3} \rightarrow \mathbb{R}^{3},
\end{equation}
defined as the ellipsoid on $\mathbb{R}^{3}$, namely
\begin{equation}
\frac{x^{2}}{d_{1}^{2}} + \frac{y^{2}}{d_{2}^{2}} + \frac{z^{2}}{d_{3}^{2}} = 1.
\end{equation}
Let \(d_{1} = (p - 1)^{1/2}, d_{2} = (q - 1)^{1/2}, d_{3} = (r - 1)^{1/2}\). Then those positive points \((x, y, z)\) on $\mathbb{R}^{3}$ such that
\begin{equation}
x = \sqrt{A}, \: y = \sqrt{B}, \: z = \sqrt{C},
\end{equation}
where the three points
\begin{eqnarray}
& &(x, y, 0),\\
& &(x, 0, z),\nonumber \\
& &(0, y, z)\\
& &x = \sqrt{A}, \: y = \sqrt{B}, \: z = \sqrt{C},
\end{eqnarray}
lie in the $XY$, $XZ$ and $YZ$ planes respectively, indicates the region of points \((x, y, z)\) on $\mathbb{R}^{3}$ for which
\begin{equation}
\gcd\left({2N \choose N}, pqr\right) = 1.
\end{equation}
The required positive real points \(x, y, z\), where \(x^{2} = A, y^{2} = B, z^{2} = C\) always are integers, lie on and outside the boundary of the ellipsoid in Eqtn. (72). The ellipse 
\begin{equation}
\frac{x^{2}}{(\sqrt{ p - 1})^{2}} + \frac{y^{2}}{(\sqrt{q - 1})^{2}} = 1,
\end{equation}
is the $XY$ trace of the ellipsoid, where for \(x^{2} = A, y^{2} = B\), 
\begin{equation}
\frac{x^{2}}{(\sqrt{ p - 1})^{2}} + \frac{y^{2}}{(\sqrt{q - 1})^{2}} \geq 1.
\end{equation}
The ellipse
\begin{equation}
\frac{x^{2}}{(\sqrt{ p - 1})^{2}} + \frac{z^{2}}{(\sqrt{r - 1})^{2}} = 1,
\end{equation}
is the $XZ$ trace of the ellipsoid, where for \(x^{2} = A, z^{2} = C\), 
\begin{equation}
\frac{x^{2}}{(\sqrt{ p - 1})^{2}} + \frac{z^{2}}{(\sqrt{r - 1})^{2}} \geq 1.
\end{equation} 
Finally the ellipse
\begin{equation}
\frac{y^{2}}{(\sqrt{ q - 1})^{2}} + \frac{z^{2}}{(\sqrt{r - 1})^{2}} = 1,
\end{equation}
is the $YZ$ trace of the ellipsoid, where for \(y^{2} = B, z^{2} = C\), 
\begin{equation}
\frac{y^{2}}{(\sqrt{q - 1})^{2}} + \frac{z^{2}}{(\sqrt{r - 1})^{2}} \geq 1.
\end{equation} 
\section{Catalan Numbers}
Let $C_{N}$ denote the $N^{th}$ Catalan number~\cite{Erdos} (See page 90),~\cite{Encyclopedia}
\begin{equation}
C_{N} = \frac{{2N \choose N}}{N + 1}.
\end{equation}
It follows that if
\begin{equation}
\gcd\left({2N \choose N}, pqr\right) = 1,
\end{equation}
then the primes \(p, q, r\) cannot be divisors of $C_{N}$ since they cannot be divisors of $N + 1$. This serves as a rule for determining at least some of those primes that do not divide the $N^{th}$ Catalan number. 

\indent In a MATLAB session the author was able to find that for the numbers,
\begin{eqnarray}
& &N = 10, N = 756, N = 757,\\
& &p = 3, q = 5, r = 7,\nonumber \\
\end{eqnarray}
for which
\begin{eqnarray}
& &\gcd\left({20 \choose 10}, 105\right) = 1,\\
& &\gcd\left({1512 \choose 756}, 105\right) = 1, \nonumber \\
& &\gcd\left({1514 \choose 757}, 105\right) = 1,
\end{eqnarray}
where
\begin{eqnarray}
10 &=&1 \cdot 3^{2} + 1\\
   &=&2 \cdot 5^{1} + 0\nonumber \\
   &=&1 \cdot 7 + 3,
\end{eqnarray}
\begin{eqnarray}
a_{2}&=&1, a_{1} = 0, a_{0} = 1,\\
b_{1}&=&2, b_{0} = 0,\nonumber \\
c_{1}&=&1, c_{0} = 3,
\end{eqnarray}
\begin{eqnarray}
756&=&1 \cdot 3^{6} + 1 \cdot 3^{3}\\
   &=&1 \cdot 5^{4} + 1 \cdot 5^{3} + 1 \cdot 5^{1} + 1\nonumber \\
   &=&2 \cdot 7^{3} + 1 \cdot 7^{2} + 3 \cdot 7,
\end{eqnarray}
\begin{eqnarray}
a_{6}&=&1, a_{5} = a_{4} = 0, a_{3} = 1, a_{2} = a_{1} = a_{0} = 0,\\
b_{4}&=&1, b_{3} = 1, b_{2} = 0, b_{1} = 1, b_{0} = 1,\nonumber \\
c_{3}&=&2, c_{2} = 1, c_{1} = 3, c_{0} = 0,
\end{eqnarray}
\begin{eqnarray}
757&=&1 \cdot 3^{6} + 1 \cdot 3^{3} + 1\\
   &=&1 \cdot 5^{4} + 1 \cdot 5^{3} + 1 \cdot 5^{1} + 2\nonumber \\
   &=&2 \cdot 7^{3} + 1 \cdot 7^{2} + 3 \cdot 7 + 1,
\end{eqnarray}
\begin{eqnarray}
a_{6}&=&1, a_{5} = a_{4} = 0, a_{3} = 1, a_{2} = a_{1} = a_{0} = 1,\\
b_{4}&=&1, b_{3} = 1, b_{2} = 0, b_{1} = 1, b_{0} = 2,\nonumber \\
c_{3}&=&2, c_{2} = 1, c_{1} = 3, c_{0} = 1.
\end{eqnarray}
We get from these,
\begin{eqnarray}
& &\gcd(C_{10}, 105) = 1,\\
& &\gcd(C_{756}, 105) = 1, \nonumber \\
& &\gcd(C_{757}, 105) = 1,
\end{eqnarray}
The previous authors show a Table~\cite{Erdos} (See page 91) of values for $A(N)$ for all \(N \in [1, 100]\). The primes \(p = 3, q = 5, r = 7\) are twin primes with a prime gap of \(5 - 3 - 1 = 7 - 5 - 1 = 1\), where \(10 \in [1, 100]\) and \(756, 757 \in [100, 1000]\). Therefore it is possible that these numbers $N$ for which ${2N \choose N}$ is not divisible by $pqr$, grow very rapidly. As \(N \rightarrow \infty\) both $C_{N}$ and the central binomial coefficient would have millions of digits and for larger and larger $N$, eventually an arbitrarily large number of digits. So it would require software with very high or arbitrary precision arithmetic functionality, to compute both $C_{N}$ and ${2N \choose N}$. It goes without saying that the use of Stirling's approximation for large $N!$ and $(2N)!$ with asymptotic limits~\cite{Abramowitz} (page 257, Formula 6.1.38),
\begin{eqnarray}
& &N! \sim N^{N}\sqrt{2\pi N}e^{-N},\\
& &(2N)! \sim (2N)^{2N}\sqrt{4\pi N}e^{-2N},
\end{eqnarray}
for the approximation of 
\begin{equation}
{2N \choose N} = \frac{(2N)!}{N!N!}, 
\end{equation}
and the approximation with asymptotic limit for large $N$,
\begin{equation}
C_{N} \sim \frac{4^{N}}{\sqrt{\pi}N^{3/2}}
\end{equation}
to find whether or not the three fixed, distinct odd primes \(p, q, r\) fail to divide ${2N \choose N}$ for \(N \gg 1000\) or $C_{N}$ would not give as much information as exact division of the actual large numbers ${2N \choose N}$, $C_{N}$ (e.g., for \(N \gg 1000\)) by this product of three primes.

\pagebreak

\bigskip
\hrule
\bigskip

\noindent
2010 \emph{Mathematics Subject Classification}: Primary 11B65; Secondary 11A41. \\
\textbf{Keywords}: Catalan number, central binomial coefficient, commensurable number, incommensurable number, integer bases.

\bigskip
\hrule
\bigskip

\noindent
Concerned with Sequence A000984.\\
Concerned with Sequence A000108.

\end{document}